\documentclass{article}
\usepackage{amssymb}
\usepackage{amsmath}
\usepackage{amsthm}
\usepackage{caption}
\usepackage{subcaption}
\usepackage{xypic}
\usepackage{graphicx}
\usepackage{pgfplots}
\usepackage{hyperref}
\usepackage{blkarray}
\theoremstyle{theorem}
\newtheorem{theorem}{Theorem}
\newtheorem{prop}[theorem]{Proposition}
\newtheorem{lemma}[theorem]{Lemma}
\newtheorem{cor}[theorem]{Corollary}
\newcommand{\Rb}{\mathbb{R}}
\newcommand{\Nb}{\mathbb{N}}
\newcommand{\Zb}{\mathbb{Z}}
\newcommand{\Fb}{\mathbb{F}}
\newcommand{\PP}{\mathcal{P}}
\newcommand{\LL}{\mathcal{L}}

\author{David Nash and Jonathan Needleman}
\title{When are finite projective planes magic?}
\begin{document}
%\thispagestyle{empty}
%\begin{center}
%\Large
% TITLE GOES HERE
%When are finite projective planes magic?
%\end{center}

%\begin{flushright}
%David Nash  \\
%Le Moyne College \\
%Syracuse, NY 13214\\
%\verb+nashd@lemoyne.edu+

%\vspace{2 mm}

%Jonathan Needleman\\
%Le Moyne College \\
%Syracuse, NY 13214\\
%\verb+needlejs@lemoyne.edu+
%\end{flushright}
\maketitle
\begin{abstract}
This article studies a generalization of magic squares to finite projective planes.  In traditional magic squares the entries come from the natural numbers.  This does not work for finite projective planes, so we instead use Abelian groups.  For each finite projective plane we demonstrate a small group over which the plane can labeled magically.  In the prime order case we classify all groups over which the projective plane can be made magic.
\end{abstract}

Magic squares have a long history,  with surviving written examples dating back to at least $300$ BC.  In fact, according to Chinese legends, a $3\times 3$ magic square that is today known as the ``Lo Shu'' square was observed as a pattern on the shell of a tortoise by Emperor Yu sometime between 2200 and 2100 BC!  Since those ancient times, magic squares have been marveled at and studied in numerous cultures the around the world.  The idea is to fill in numbers into a square so that the sum along each row, column and diagonal are all equal to the same number -- often called the magic constant.  For instance the square in figure \ref{fig:square} is a representation of the ``Lo Shu'' magic square from the legends.  Here the magic constant is 15.

\begin{figure}[!htbp]
\captionsetup[subfigure]{labelformat=simple}
\centering
\begin{tikzpicture}
    \draw[thick] (0,0) -- (2,0);
    \draw[thick] (0,.667) -- (2,.667);
    \draw[thick] (0,1.333) -- (2,1.333);
    \draw[thick] (2,2) -- (0,2);
    \draw[thick] (0,0) -- (0,2);
    \draw[thick] (.667,0) -- (.667,2);
    \draw[thick] (1.333,0) -- (1.333,2);
    \draw[thick] (2,0) -- (2,2);
    \draw (.333,.333) node {$8$};
    \draw (1,.333) node {$1$};
    \draw (1.667,.333) node {$6$};
    \draw (.333,1) node {$3$};
    \draw (1,1) node {$5$};
    \draw (1.667,1) node {$7$};
    \draw (.333,1.667) node {$4$};
    \draw (1,1.667) node {$9$};
    \draw (1.667,1.667) node {$2$};
\end{tikzpicture}
\caption{The ``Lo Shu'' magic square.}
\label{fig:square}
\end{figure}

There have been numerous generalizations of magic squares to other shapes.  Ely  introduces the idea of magic designs \cite{Ely}.  A design is just a set of ``points'' and a set of ``lines,'' with each line being a  subset of points.    A magic design is then an injective function from the points to the natural numbers where the sum along any line is constant.  While Ely  focused primarily on designs based on triangles and hexagons other designs have since been studied.  A particulary nice family of designs comes from (combinatorial) configurations.  These are designs where every line has the same number of points, and every point has the same number of lines through it.   Magic stars are an example with two lines through every point \cite{Trenkler:2} and more recently Raney studied magic configurations where three lines pass through every point, and each line contains three points \cite{Raney}.

Projective planes are particularly nice configurations because the number of points on a line is the same as the number of lines through a point.  Unfortunately, we will show that finite projective planes are never magic for any subset of the integers.  Thankfully, there is no reason to limit ourselves to integers.  All that is really needed to discuss ``magicness'' is the ability to add the entries along a line.  Because of this we can try to make projective planes magic over Abelian groups.

For the special case of the Fano Plane (the finite projective plane of order 2), Miesner and the first author \cite{MN} show that no magic labelings exist with labels in $\Zb/n\Zb$ for any $n$.  In this paper we further generalize and study the ``magicness'' of all finite projective planes.  Specifially, for every finite projective plane we will find a group for which it is magic and for a certain class of projective planes we will classify all groups for which it can be made magic.

For the interested reader, the authors further generalize the notion of magicness over Abelian groups to higher-dimensional finite projective spaces in \cite{NN}.

%%%%%%%%%%%%%%%%%%%%%%%%
\section{Projective planes}\label{proj_plane}
%%%%%%%%%%%%%%%%%%%%%%%%
We summarize some standard results about finite projective planes.  For a general resource on finite projective planes the authors suggest \cite{AS} as an introduction and \cite{Dembowski} for more advanced readers.

A projective plane $\Pi=(\PP,\LL)$ is made up of a set of points $\PP$ and lines $\LL$.  For any point $x \in \PP$ we let $\LL^x$ denote the set of all lines through $x$.  To be a projective plane the following three axioms must hold.
\begin{enumerate}
    \item There is a line between every pair of points. -- If $x, x'\in \PP$ with $x\neq x'$ then there is a unique $L\in \LL$ with $x,x'\in L$.
    \item Each pair of lines intersects at a unique point. -- If $L, L'\in \LL$ with $L\neq L'$ then there is a unique $x\in \PP$ with $x\in L\cap L'$.
    \item The plane contains a quadrilateral.  -- There exist $x_1, x_2, x_3, x_4\in \PP$ so that there is no $L\in \LL$ that contains three of the points.
\end{enumerate}
A finite projective plane $\Pi=(\PP,\LL)$ is just a plane where the number of points, $|\PP|$, is finite.  For finite projective planes the axioms imply some basic facts.

\textsc{Fact.} {\em For a finite projective plane $\Pi=(\PP,\LL)$ there exists a number $n \in \Nb$, called the \emph{order} of $\Pi$, and the following hold:
    \begin{enumerate}
        \item Every line contains $n+1$ points.  -- $|L|=n+1$ for all $L\in \LL$.
        \item Every point is on $n+1$ lines.  -- $|\LL^x|=n+1$ for every $x\in \PP$.
        \item There are an equal number of points and lines.  -- $|\PP|=|\LL|=n^2+n+1$.
    \end{enumerate}}

It is an open question to classify for which orders $n$ there exists a projective plane, however, it is known that for any prime $p$ and any $k \in \Nb$, there exists a projective plane of order $n=p^k$.  In fact, these are the only orders for which projective planes are known to exist.  However, finite projective planes have been completely classified for small orders and for sufficiently small order ($\leq 8$) they are all constructible in a uniform way.  We will make use of that construction to deal with small order cases and thus we give it here.

Let $\Fb_q$ be a finite field of order $q=p^k$ for some prime $p$.  We construct a projective plane $\Pi_q=(\PP_q,\LL_q)$ in the following way. View $\Fb_q^3$ as a vector space over $\Fb_q$, then let $\PP_q=\{\text{1-dim subspaces of } \Fb_q^3\}$ and $\LL_q=\{\text{2-dim subspaces of } \Fb_q^3\}$.   One can verify that this construction yields a finite projective plane of order $q$.  Since a point in $\Pi_q$ is a line through the origin of $\Fb_q^3$ we can describe the points as follows.  Given $\langle x_1, x_2, x_3\rangle\in\Fb_q^3$ a non-zero vector, the set $[x_1, x_2, x_3]=\{\langle cx_1, cx_2, cx_3\rangle \mid c\in \Fb_q^*\}$ describes all points in $\PP_q$.  The lines in $\LL_q$ are the planes through the origin in $\Fb_q^3$.  Any vector $v$ determines a plane through the origin by considering all vectors orthogonal to $v$.  So the lines $\LL_q$ are described as $[[x_1, x_2, x_3]]=\{u \in \Fb_q^3 \mid u \cdot \langle x_1, x_2, x_3\rangle = 0\}$, where  $\langle x_1, x_2, x_3\rangle\in\Fb_q^3$ is a non-zero vector.  Notice, for $c\in \Fb_q\backslash\{0\}$, $[[cx_1, cx_2, cx_3]]=[[x_1,x_2,x_3]]$.

The smallest possible case when $q=2$ is a plane with 7 points and 7 lines that is usually called the Fano plane, see figure~\ref{fig:fano}.  There are 3 points on each line and 3 lines through each point.  The Fano plane is actually the unique finite projective plane of order 2.

\begin{figure}[ht!]
\captionsetup[subfigure]{labelformat=simple}
\centering
\begin{tikzpicture}
    \draw[thick] (1,0) -- (5,0);
    \draw[thick] (5,0) -- (3,3.464);
    \draw[thick] (1,0) -- (3,3.464);
    \draw[thick] (3,0) -- (3,3.464);
    \draw[thick] (2,1.732) -- (5,0);
    \draw[thick] (1,0) -- (4,1.732);
    \draw[thick] (3,1.155) circle(1.155);
    \draw[thick,fill] (1,0) circle (0.1);
    \draw[thick,fill] (2,1.732) circle (0.1);
    \draw[thick,fill] (3,0) circle (0.1);
    \draw[thick,fill] (3,1.155) circle (0.1);
    \draw[thick,fill] (3,3.464) circle (0.1);
    \draw[thick,fill] (4,1.732) circle (0.1);
    \draw[thick,fill] (5,0) circle (0.1);
    \draw(-0.1,0) node {\small $x_1=(1,0,0)$};
    \draw (0.9,1.732) node {\small $x_6=(1,1,0)$};
    \draw (3,-0.5) circle node {\small $x_5=(1,0,1)$};
    \draw (3.7,1.155) node {\tiny $(1,1,1)$};
    \draw (2.6,1.155) node {\small $x_7$};
    \draw (3,3.964) node {\small $x_2=(0,1,0)$};
    \draw (5.1,1.732) node {\small $x_4=(0,1,1)$};
    \draw (6.1,0) node {\small $x_3=(0,0,1)$};
\end{tikzpicture}
\caption{The Fano plane, $\Pi_2$.}
\label{fig:fano}
\end{figure}
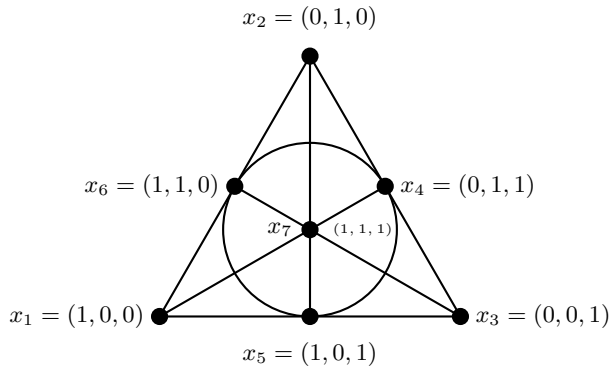

%%%%%%%%%%%%%%%%%%%%%%%%%%
\section{Non-magicness}
%%%%%%%%%%%%%%%%%%%%%%%%%%
Classically, an $n\times n$ square is magic if it is labeled with the numbers $\{1,\ldots, n^2\}$ so that each row, column, and diagonal sum to the same value.  We can ask a similar question for a finite projective plane $\Pi=(\PP,\LL)$ of order $n$.   We would like to assign the values $\{1,\ldots , n^2+n+1\}$ to the points $\PP$ so that the sum along any line is the same.  Unfortunately, this is impossible, not only for the numbers $\{1,\ldots, n^2+n+1\}$, but for any set of $n^2+n+1$ distinct real numbers.

To prove this, we need the \emph{incidence matrix} $A$ of the projective plane $\Pi$.  Let $x_1, \ldots , x_{n^2+n+1}$ be an enumeration of the points $\PP$ and $L_1, \ldots, L_{n^2+n+1}$ be an enumeration of the lines $\LL$.  The rows of $A$ will be indexed by the lines of $\Pi$ and the columns by the points.  Given $L_i\in \LL$ and $x_j\in\PP$, the entry $A_{i,j}$ is $1$ if the point $x_j$ is on the line $L_i$, and $0$ otherwise.  For example, working from Figure~\ref{fig:fano}, if we take $L_1=\overline{x_2x_3}$, $L_2=\overline{x_1x_3}$, $L_3=\overline{x_1x_2}$, $L_4=\overline{x_1x_4}$, $L_5=\overline{x_2x_5}$, $L_6=\overline{x_3x_6}$, and $L_7=\overline{x_4x_5}$, then the incidence matrix for Fano Plane is as follows:
$$
\begin{blockarray}{cccccccc}
&x_1&x_2&x_3&x_4&x_5&x_6&x_7\\
\begin{block}{c[ccccccc]}
L_1&0&1&1&1&0&0&0\\
L_2&1&0&1&0&1&0&0\\
L_3&1&1&0&0&0&1&0\\
L_4&1&0&0&1&0&0&1\\
L_5&0&1&0&0&1&0&1\\
L_6&0&0&1&0&0&1&1\\
L_7&0&0&0&1&1&1&0\\
\end{block}
\end{blockarray}
 $$

An important observation to make here is that the incidence matrix for any finite projective plane will be invertible over $\Rb$.  A beautiful way to prove this fact is to consider the matrix $AA^T$.  Observe, the $i,j$-entry of the matrix $AA^T$ is exactly the number of points on the intersection of the lines $L_i$ and $L_j$.  Since there are $n+1$ points on every line, we can see that the diagonal entries of $AA^T$ are all equal to $n+1$.  In addition, distinct lines always intersect at a single point and hence all of the other entries in $AA^T$ are equal to 1.  One can then use row reduction to show that $\det(AA^T) = (n+1)^2n^{n^2+n}$ (see e.g. \cite{Dembowski}).  Thus $AA^T$, and more importantly $A$ itself, is invertible.

The incidence matrix also gives us a natural way of translating from a labeling of the points in $\PP$ to the sums of those labels along each line in $\LL$. To assist in this translation we view labelings as functions on the set of points, $\PP$.  Let $f:\PP\to\Rb$ denote a function so that the sum along each line in $\LL$ is the same magic constant $c \in \Rb$.  If such an $f$ exists which is also injective then we say that $\Pi$ is \emph{magic over} $\Rb$.  The magic constant condition can then be represented by the matrix equation $A\mathbf{f} = \mathbf{c}$ where $A$ is the incidence matrix for $\Pi$, $\mathbf{f}$ is the column vector $[f(x)]_{x\in\PP}$, and $\mathbf{c}$ is the vector with all entries equal to $c$.  Using this data, we demonstrate that the plane $\Pi$ is not magic over $\Rb$.

\begin{prop}\label{RNotMagic} No finite projective plane $\Pi=(\PP,\LL)$ is magic over $\Rb$.
\end{prop}

\begin{proof}
Let $f:\PP\to \Rb$ be any real valued function on the points $\PP$.  $\mathbf{f}$ is a column vector with the rows indexed by $\PP$.  If $A$ is the incidence matrix for $\Pi$ then $A\mathbf{f}$ will be a column vector with the rows indexed by $\LL$.  Hence we may think of $A\mathbf{f}$ as a real-valued function on $\LL$.  The value of the row indexed by $L\in\LL$ is exactly $\sum_{x\in L}f(x)$, the sum of $f$ along $L$, by construction of $A$.

We are interested in when $A\mathbf{f}$ is a constant function.  However, as mentioned above, the incidence matrix is invertible.  This means that the equation $A\mathbf{f}=c$ has a unique solution which, therefore, must be the constant function $f(x)=\frac{c}{n+1}$ for all $x \in \PP$.  Since $\mathbf{f}$ is not injective, it follows that $\Pi$ is not magic over $\Rb$.
\end{proof}

In light of the result above, we seek to find conditions under which a finite projective plane $\Pi=(\PP,\LL)$ can be considered to be magic.  We must be able to add the values assigned to points and the addition must be commutative since the points are not in any particular order.  So we let $G$ be an Abelian group and consider a $G$-valued function $v:\PP\rightarrow G$ on the points of a projective plane.   For any subset $S\subset \PP$ we then define $v(S)=\sum_{x\in S}v(x)$.  The function $v$ is called \emph{line invariant} if $v(L)=v(L')$ for all $L, L'\in \LL$.  When this holds we call $v(L)$ the \emph{magic constant}.

The set of line invariant functions has a really beautiful structure, but unfortunately, it includes the constant functions which are trivially line invariant because every line has the same number of points.   Since constant functions do not get at the nature of magicness, we will only refer to $v$ as a \emph{pseudomagic} function when it is both line invariant and non-constant.   If, furthermore, $v$ is actually injective, then we will call $v$ \emph{magic}.

We say a finite projective plane $\Pi=(\PP,\LL)$  is \emph{pseudomagic over an Abelian group $G$} (resp.\ {\emph{magic over $G$}) and $G$ \emph{admits a pseudomagic} (resp. \emph{magic}) \emph{function} $v$ if and only if there exists a pseudomagic (resp.\ magic) function $v:\PP\rightarrow G$.  As we will show in Theorem~\ref{thm:torsionfree}, $\Pi$ will not admit a pseudomagic function over a group $G$ unless that group contains elements of finite order.  Groups that do not contain any elements of finite order are called \emph{torsion-free} groups.

\begin{theorem}\label{thm:torsionfree}
Let $G$ be an Abelian torsion-free group and let $\Pi=(\PP,\LL)$ be a finite projective plane.  Then $\Pi$ is not pseudomagic over $G$.
\end{theorem}

\begin{proof}
Let $v$ be a line invariant $G$-valued function on $\PP$.  Let $H$ be the subgroup $\langle v(x)\mid x\in\PP\rangle$ of $G$.  $H$ is a finitely generated torsion-free Abelian group and thus $H\cong\Zb^k$ for some $k\in \Nb$.  We may therefore view $v$ as a $\Zb^k$-valued function, and we may write
  $$v=\bigoplus_{i=1}^k v_i$$
where each $v_i:\PP\rightarrow \Zb$.  Since $v(L)$ is independent of $L\in\LL$ it also holds that $v_i(L)$ is independent of $L$ for each $i$.  By Proposition~\ref{RNotMagic} each $v_i$ is a constant function and hence $v$ is constant on $\PP$ as well.
\end{proof}

%%%%%%%%%%%%%%%%%%%%%%%%%%
\section{Magicness}
%%%%%%%%%%%%%%%%%%%%%%%%%%
In order to find a group over which a projective plane is magic we must look at torsion groups.  Cyclic groups are a natural place to start.  We begin by classifying cyclic groups that admit a pseudomagic function for a given projective plane. Throughout the section let $\Pi=(\PP,\LL)$ be a projective plane of order $n$ and let $m \in \Nb$.

One way to attempt to find a pseudomagic function from $\PP$ to $\Zb/m\Zb$ is to create a function that is constant on a chosen line $L$ and zero on all points off of $L$.  As it turns out, if the constant is chosen carefully then the function will be line invariant.  More precisely, given any line $L \in \LL$, we define the function $v_{L}: \PP \to \Zb/m\Zb$ as follows:
\begin{equation}\label{vL}
v_{L}(x) = \left\{\begin{tabular}{lr} $\frac{m}{(n,m)}$ & if  $x \in L$ \\ 0 & if $x \not \in L$ \end{tabular}\right.
\end{equation}
Notice, when $m \mid n$, $v_{L}$ is the characteristic function of $L$.  As an example, consider the Fano plane and the group $G=\Zb/6\Zb$.  We may choose the line $L_1$, then the function $v_{L_1}$ would correspond to labeling the points on that line by $3$, see figure~\ref{fig:pseudomagic}.}

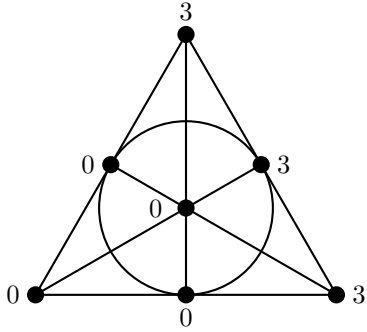
\begin{figure}[ht!]
\captionsetup[subfigure]{labelformat=simple}
\centering
\begin{tikzpicture}
    \draw[thick] (1,0) -- (5,0);
    \draw[thick] (5,0) -- (3,3.464);
    \draw[thick] (1,0) -- (3,3.464);
    \draw[thick] (3,0) -- (3,3.464);
    \draw[thick] (2,1.732) -- (5,0);
    \draw[thick] (1,0) -- (4,1.732);
    \draw[thick] (3,1.155) circle(1.155);
    \draw[thick,fill] (1,0) circle (0.1);
    \draw[thick,fill] (2,1.732) circle (0.1);
    \draw[thick,fill] (3,0) circle (0.1);
    \draw[thick,fill] (3,1.155) circle (0.1);
    \draw[thick,fill] (3,3.464) circle (0.1);
    \draw[thick,fill] (4,1.732) circle (0.1);
    \draw[thick,fill] (5,0) circle (0.1);
    \draw(0.7,0) node {$0$};
    \draw (1.7,1.732) node {$0$};
    \draw (3,-0.3) circle node {$0$};
    \draw (2.6,1.155) node {$0$};
    \draw (3,3.764) node {$3$};
    \draw (4.3,1.732) node {$3$};
    \draw (5.3,0) node {$3$};
\end{tikzpicture}
\caption{The pseudomagic function $v_{L_1}$ from the points in the Fano Plane to $\Zb/6\Zb$.}
\label{fig:pseudomagic}
\end{figure}

Observe that the sum along any line is then exactly $3 \in \Zb/6\Zb$.  We now show more generally that the functions $v_{L}$ for each $L \in \LL$ are always line invariant.  This fact relies directly on the incidence structure of finite projective planes.

\begin{lemma}\label{magicline}
$v_L: \PP \to \Zb/m\Zb$ is line invariant.
\end{lemma}

\begin{proof}
Observe that we have
	\begin{equation} \label{std_func}
        v_L(L) = \sum_{x \in L} \frac{m}{(n,m)} \equiv \frac{(n+1)m}{(n,m)} \equiv \frac{nm}{(n,m)}+\frac{m}{(n,m)}\mod m.
	\end{equation}
However $m \mid \frac{nm}{(n,m)}$ so $v_L(L)=\frac{m}{(n,m)}\mod m$.  Let  $L'\in \LL$ be any other line.  We know that $L'$ intersects $L$ in exactly one point, $x$.  Thus, $v_L(x)=\frac{m}{(n,m)}$ and $v_L(x')=0$ for all other $x'\in L'$. It follows that $v_L(L')=\frac{m}{(n,m)}$ as well.
\end{proof}

\begin{cor}\label{nonconstant} If $(n,m)>1$, then $\Pi$ is pseudomagic over $\Zb/m\Zb$.
\end{cor}

\begin{proof}
$v_L=0$ if and only if $(n,m)=1$.  Hence when $(n,m) \neq 1$, the function $v_L$ is a pseudomagic function on $\PP$ for each line $L \in \LL$.
\end{proof}

It turns out these are the only cyclic groups for which $\Pi$ can be made pseudomagic.

\begin{prop}\label{cyclic_magic} $\Pi$ is pseudomagic over $\Zb/m\Zb$ if and only if $(n,m)\neq 1$. Furthermore, $\Pi$ is never magic over $\Zb/m\Zb$.
\end{prop}

\begin{proof}
The ``if" is proven in Corollary~\ref{nonconstant}. For the other direction assume there is a line invariant function $v: \PP \to \Zb/m\Zb$ with magic constant $g\in \Zb/m\Zb$.  Let $a, b\in \PP$ and let $L=\overline{ab}$ be the line containing $a$ and $b$ and let $L^c$ be the set of points not on $L$.  Each point of $L^c$ is on exactly one line in the set of lines $\LL^a\backslash\{L\}$ (recall $\LL^a$ is all lines through $a$).  Therefore,
	\[ v(L^c)=\sum_{L'\in \LL^a\backslash\{L\}} (v(L')-v(a))= ng-nv(a), \]
since there are $n$ lines in $\LL^a\backslash\{L\}$ and $a$ is not in $L^c$.  Similarly, $v(L^c)=ng-nv(b)$ and hence $nv(a)=nv(b)$ in $\Zb/m\Zb$.  Since $a, b$ were arbitrary, we may conclude that $nv(x)$ is independent of $x \in \PP$.  Therefore, $n(v(x)-v(y))=0$ for all $x, y\in \PP$.  Hence, for all $x, y\in \PP, v(x)-v(y)$  is in the kernel of the homomorphism $\phi_n: t\mapsto nt$ in $\Zb/m\Zb$, and so each $v(x)$ is in the same coset of the kernel.  The homomorphism has $|\text{Ker}(\phi_n)|=(n,m)$, and so a line invariant function $v:\PP \mapsto \Zb/m\Zb$ can take on up to $(n,m)$ different values.  Thus, when $(n,m)=1$, $v$ is a constant function and not a pseudomagic function.  Furthermore, $(n,m)<n^2+n+1=|\PP|$, so $v$ can never be magic.
\end{proof}

From Corollary~\ref{nonconstant}, it is not hard to find a group $G$ which admits a magic function on $\Pi$.  Take one copy of $\Zb/n\Zb$ for every line in $\LL$.  That is let $G=(\Zb/n\Zb)^k$ where $k=|\LL|=n^2+n+1$.  Next, choose an enumeration of $\LL=\{L_1, L_2,\ldots, L_k\}$ and define $v:\PP \rightarrow  (\Zb/n\Zb)^k$ as $v(x)=(v_{L_1}(x), v_{L_2}(x),\ldots, v_{L_k}(x))$.
By Corollary~\ref{nonconstant} this is a pseudomagic function.  However, it is also magic because for any two points $x_1, x_2\in L$ there is a line that contains one of them but not the other.

This construction seems inefficient as $|G|$ is much larger than $|\PP|$ and the exponent $k$ depends on $\Pi$.  Next, we find the smallest $r$ so that $\Pi$ is magic over $(\Zb/n\Zb)^r$.  Since $|\PP|=n^2+n+1$ we know $r\geq 3$, and in fact, we show $r=3$ works.  Our general proof relies on having $n \geq 5$ and hence we treat the cases $n=2$, $3$, and $4$  separately.

\begin{theorem}\label{MinimalMagicGroup}
If $\Pi=(\PP,\LL)$ is a projective plane of order $n\geq 5$ then $\Pi$ is magic for the group $G=(\Zb/n\Zb)^3$.
\end{theorem}

\begin{proof}
The plan is to create three separate pseudomagic functions $v_1$, $v_2$, and $v_3$ from $\PP$ to $\Zb/n\Zb$ which together define a magic function $(v_1, v_2, v_3): \PP \to G$.

To begin, we label some points and lines for reference.  Let $x$ be any point in $\PP$ and let $L_0, L_1, \dots, L_n$ be an enumeration of the $n+1$ lines through $x$.  Next, let $y$ be another point on $L_n$ and let $L'_0, \dots, L'_{n-1}$ denote the other $n$ lines through $y$.  Then, for each $1 \leq i,j \leq n-1$, we let $w_{i,j}=L_i\cap L_j'$ as in figure~\ref{fig:pqs}.   Finally, let $z_1, z_2, \dots, z_{n-1}$ denote the points in $L_n \setminus \{x, y\}$.  For now the choice of the $z_k$ is arbitrary, but later we will be more specific in our labeling.

\begin{figure}[!htbp]
\scalebox{1.15}{
    \begin{subfigure}[b]{.4\linewidth}
        \[\xy
        (0,0)*{};(30,0)*{}**\dir{-};(15,26)*{};(30,0)*{}**\dir{-};(15,26)*{};(0,0)*{}**\dir{-};
        (0,0)*{\bullet};(-2,-2.5)*{w_{0,0}};(30,0)*{\bullet};(32,-2.5)*{y};(15,26)*{\bullet};(15,28.5)*{x};
        (10,0)*{};(15,26)*{}**\dir{-}; (10,-2.5)*{w_{i,0}};(10,0)*{\bullet};
        (30,0)*{};(5,8.6)*{}**\dir{-};(5,8.6)*{\bullet};(3.5, 10.8)*{w_{0,j}~~~};
        (11.3,6.3)*{\bullet};(15,7)*{w_{i,j}};
        \endxy\]
        \caption{}
        \label{fig:pqs}
    \end{subfigure}
    \begin{subfigure}[b]{.4\linewidth}
        \[\xy
        (0,0)*{};(30,0)*{}**\dir{-};(15,26)*{};(30,0)*{}**\dir{-};(15,26)*{};(0,0)*{}**\dir{-};
        (0,0)*{\bullet};(-2,-2.5)*{w_{0,0}};(30,0)*{\bullet};
        (30,0)*{};(10,17.3)*{}**\dir{-};(10,17.3)*{\bullet};(8.5,19.3 )*{w_{0,1}~~~~};(32,-2.5)*{y};
        (0,0)*{};(20,17.3)*{}**\dir{-};(20,17.3)*{\bullet};(21.5, 19.3)*{~z_2};
        (15,13)*{\bullet};(15,16.5)*{w_{1,1}};
        (15,0)*{\bullet};(15,-2.5)*{w_{1,0}};
        (7.5,13)*{\bullet};
        (7.5,13)*{};(20,17.3)*{}**\crv{~*=<2pt>{.}(16.5,-15)};(19,9.6)*{\bullet};
        (3.75,6.5)*{\bullet};(22.5,13)*{\bullet};(27,14)*{~z_{n-2}};
        (3.75,6.5)*{};(22.5,13)*{}**\crv{( 19,-9)};(21.3,7.5)*{\bullet};
        (22.5,2.5)*{w_{h,1}};(1.25,7)*{w_{0,h'}~~~~};
        (0,0)*{};(25,8.8)*{}**\dir{-};(25,8.8)*{\bullet};(27.5,9)*{~z_1};
        \endxy\]
        \caption{}
        \label{fig:rk}
    \end{subfigure}
    }
    \caption{Labeling $\PP$}
    \label{fig:labels}
\end{figure}

From this point forward a couple of minor details in the proof depend on the parity of $n$.  The main technical difference stems from the fact that the sum of all of the elements in $\Zb/n\Zb$ is $0$ when $n$ is odd, but is $\frac{n}{2}$ when $n$ is even.  As it turns out, our construction is unaffected by this difference, but for simplicity we choose to deal with $n$ odd first, and then explain why it works for $n$ even as well.

For $n$ odd define the first two pseudo-magic functions as follows:
\begin{equation}\label{V1V2} v_1 = \sum_{k=1}^n (k-1)v_{L_k} \qquad v_2 = \sum_{k=1}^n (k-1)v_{L'_k}
\end{equation}
Since $x$ is on each $L_k$ we have $v_1(x)=\sum_{k=1}^n (k-1) = 0$ in $\Zb/n\Zb$.  Similarly, $v_2(y)=0$.  The function $v':=(v_1, v_2): \PP \to (\Zb/n\Zb)^2$ then already has unique values on most of the points. The only equalities are the following:
\begin{equation}\label{equalities}\begin{array}{ll}
v'(w_{0,0})=v'(w_{0,1})=v'(w_{1,0})= v'(w_{1,1})= (0,0) & \\
v'(w_{0,i})=v'(w_{1,i}) = (0,i-1) & 2\leq i\leq n-1\\
v'(w_{i,0})=v'(w_{i,1}) = (i-1, 0) & 2\leq i\leq n-1\\
v'(z_i)=v'(z_j) = (n-1,n-1) & 1\leq i,j \leq n-1
\end{array}
\end{equation}
since $v'(w_{i,j})=(i-1,j-1)$ for $1 \leq i, j\leq n-1$. 

We now create a third function, $v_3$, to distinguish points which have equal values on $v'$.  This is a delicate construction which requires a careful ordering of the points $z_k$ with respect to other points on the plane. Let $z_{2}=L_n\cap \overline{w_{0,0}w_{1,1}}$ and  let $J'$ be the line $\overline{w_{1,0} z_2}$.  Next, choose an $h$ so that $w_{h,1}\in L_1'\backslash\{w_{0,1}, y, w_{1,1}, J'\cap L'_1\}$.  This is possible since $n\geq 5$ and so there must be at least $6$ points on a line.  Now let $z_1=\overline{w_{0,0}w_{h,1}}\cap L_n$, and $z_{n-2}=\overline{w_{1,0}w_{h,1}}\cap L_n$.  Once again, since $n\geq 5$, $z_{n-2}\neq z_2$.
Label the remaining points on $L_n$ as $\{z_3, \ldots z_{n-3}, z_{n-1}\}$, and let $L''_i=\overline{w_{0,0}z_i}$ for $1 \leq i \leq n-1$.  Define the line $J=\overline{w_{1,0}z_{n-2}}$.  For future reference we will let $w_{0,h'}=L_0\cap J$.  Figure~\ref{fig:rk} depicts the specific labeling that we have described with $J'$ dashed and $J$ solid.

We now define $v_3:\PP\rightarrow \Zb/ n\Zb$.
\begin{equation}\label{V3}
v_3=\sum_{k=1}^{n-1}(k-1)v_{L''_k}+(n-1)v_{L_0}+2v_J
\end{equation}
All that is left is to check that $v_3$ differentiates the points that had equal values under $v'$.  First, check the $z_k$.  Observe that $v_3(z_k)=k-1$ if $1\leq k\leq n-1$ with $k\neq n-2$,  and $v_3(z_{n-2})=n-1$.  The difference for $z_{n-2}$ is the inclusion of $2v_J$ in the definition of $v_3$.

Next, check the pairs $w_{i,0}$ and $w_{i,1}$.  For $2\leq i\leq n-1$, $v_3(w_{i,0})=0$.  For each of those $i$ except $h$ there is a unique $j$ with $2\leq j\leq n-1$ so that $w_{i,1}\in L''_j$.  In these cases $v_3(w_{i,1})=j-1\neq 0$.  For $h$, $w_{h,1}=L''_1\cap J$, so $v_3(w_{h,1})=2$.

Now,  check the pairs $w{0,1}$ and $w_{1,i}$.  For each $1\leq i\leq n-1$ there is a unique $j$ with $1\leq j\leq n-1$ so that $w_{1,i}\in L''_j$  we therefore have $v_3(w_{1,i})=j-1\neq n-1$. Recall, $h'$  is defined so that $w_{0,h'}=J\cap L_0$. For $1\leq i\leq n-1$ with $i\neq h'$ we have $v_3(w_{0,i})=n-1$, and $v_3(w_{0,h'})=1$.    We need to show $v_3(w_{1,h'})\neq 1$.  We know $v_3(w_{1,1})=1$ since $w_{1,1}\in L_2''$.
If $h'=1$ then  $J=\overline{w_{0,1}w_{h,1}}$, but $w_{1,0}\in J$ and $w_{1,0} \not \in \overline{w_{0,1}w_{h,1}}$, so $h'\neq 1$.

Finally, check  $w_{0,0}$, $w_{0,1}$, $w_{1,0}$, and $w_{1,1}$.  Observe that $v_3(w_{0,1})=n-1$, $v_3(w_{1,0})=2$, $v_3(w_{1,1})= 1$, and $v(w_{0,0})=0$.

Thus $v=(v_1, v_2, v_3):\PP\rightarrow (\Zb/n\Zb)^3$ is injective and hence magic when $n$ is odd.  In the case when $n$ is even define $v_1, v_2, v_3$ in the same fashion, but now $v_1(x)=v_2(y)=v_3(w_{0,0})=n/2$.  Since $n \geq 5$ we have $2<\frac{n}{2} < n-1$ and thus these changes do not impact the proof.
\end{proof}

We now treat the cases $n=2$, 3, and 4 separately.

\begin{lemma}\label{n2to4} The projective planes of order  $n=2$, $3$, and $4$ are magic over $G=(\Zb/n\Zb)^3$.
\end{lemma}

\begin{proof}
For each $2\leq n\leq 4$ there is a unique projective plane of order $n$ \cite{MacInnes} and  hence we may use the construction of the plane $\Pi_n$ described earlier.  Recall, that each point $a\in \PP_n$ is of the form $a=[x, y, z]$ with $x, y, z\in \Fb_n$.  In the $n=2$ case we have the Fano Plane (see Figure~\ref{fig:fano}) and the function
$$v(a)=(x,y,z)$$
is magic.  In the $n=3$ case, define the following functions
$$
    \begin{array}{rcl}
    v_1(a) &=&x^2+z^2+xy+2yz+2xz\\
    v_2(a) &=& y^2+x^2+yz+2zx+2yx\\
    v_3(a) &=&  z^2+y^2+zx+2xy+2zy
    \end{array}
$$
 and then $v=(v_1, v_2, v_3)$ is magic.

These functions are well defined on the sets $[x,y,z]$,  since the components are homogeneous polynomials of degree $n-1$ and for any $c\in \Fb_n$ we have $c^{n-1}=1$.  One can directly check these that these functions are magic.

For $n=4$ we need a significant modification of our construction from Theorem~\ref{MinimalMagicGroup}.  We are working with $\Fb_4=\{0,1,\alpha, \alpha+1\}$ where $\alpha^2=\alpha+1$ and we set $a_0=0$, $a_1=1$, $a_2=\alpha$, and $a_3=\alpha+1$.   We then let $w_{i,0} = [1,0,a_i]$, $w_{0,i}=[1,a_i,0]$, and $z_i=[0,1,a_i]$ for $0\leq i \leq 3$, and we let $y=[0,0,1]$ and $x=[0,1,0]$.  Next we define lines $L_i = \overline{xw_{i,0}}$, $L'_i = \overline{yw_{0,i}}$, $L''_i = \overline{w_{0,0}z_i}$, and the rest of the points $w_{i,j}$ as in Theorem~\ref{MinimalMagicGroup}.  Using these lines we again define $v_1$ and $v_2$ as in (\ref{V1V2}).  This leaves us with the same equalities described in (\ref{equalities}).  We define the line $J=\overline{y_1 w_{2,1}}$ and then set
$$
v_3=3v_{L_0}+v_{L''_2} + 2v_{L''_3} +v_J.
$$
One can then directly check that $v=(v_1, v_2, v_3)$ is magic.
\end{proof}

To this point we have a relatively small group for which a projective plane is magic, however in the case when $n$ is a prime we can say more.  It turns out that in some sense $(\Zb/n\Zb)^3$ is the only group for which $\Pi$ is magic.

\begin{theorem}
If $n$, the order of $\Pi$, is prime and $\Pi$ is magic over some Abelian group $G$,  then $(\Zb/n\Zb)^3$ is a subgroup of $G$.
\end{theorem}

\begin{proof}
Let $v:\PP \to G$ be a magic function.  We may assume $G$ is finitely generated since we may work with the subgroup generated by $\text{Im }v$.  Since $G$ is finitely generated there exist $n_1, \dots, n_k \in \Nb$ such that $G=\bigoplus_{i=1}^k \Zb/n_i\Zb$.  Thus $v=\bigoplus_{i=1}^k v_i$, where $v_i:\PP \to \Zb/n_i\Zb$ is the natural projection of $v$ into $\Zb/n_i\Zb$, and each $v_i$ is line invariant.  Without loss of generality we assume each $v_i$ is non-constant and hence pseudomagic.  Since $n$ is prime,  Proposition~\ref{cyclic_magic} implies that $n \mid n_i$, and furthermore, the proof shows $|\text{Im }v_i|\leq (n,n_i)=n$.  This means $k\geq 3$ as otherwise $|\text{Im } v|\leq n^2 <|\PP|$ which contradicts $v$ being injective.
\end{proof}

When $n$ is not prime it might be possible for smaller groups to admit a magic function.  If $m \mid n$ and $m>1$ then $mn^2>n^2+n+1 = |\PP|$, so $\Pi$ could potentially be magic over a group of order $mn^2$.

\textsc{Open Questions.}

1. {\em When $m \mid n$ with $m>1$, is $\Pi$ magic over $(\Zb/n\Zb)^2 \times \Zb/m\Zb$?}

As it turns out, for planes of the form $\Pi_q$ when $q$ is prime, every line invariant function to $\Zb/m\Zb$ is a linear combination of the $v_L$ functions.  However, this is not true when $q$ is not prime (see \cite{NN}).  In the non-prime case there are special line invariant functions which are not linear combinations of the $v_L$'s.  Our proof of Theorem~\ref{MinimalMagicGroup} does not make use of these special functions.  Therefore, it seems reasonable to expect that if our open question is answered in the affirmative, then the proof will require the special functions.

Another question that one could ask is whether or not all magic functions to larger groups $G$ actually come from functions to $(\Zb/n\Zb)^3$.  A more technical way to describe this is as follows:

2. {\em Suppose that $v: \Pi \to G$ is a magic function and $\Pi$ is order $n$, under what conditions is there a surjection $f:G \to (\Zb/n\Zb)^3$ such that $f \circ v: \Pi \to (\Zb/n\Zb)^3$ is still magic?}

\paragraph*{Acknowledgment}
The second author would like to thank Michael Raney for originally introducing him to magic configurations.

\end{document}